\documentclass[12pt,english]{amsart}

\usepackage[all,cmtip]{xy}
\usepackage{todonotes}

\usepackage{bbm}
\usepackage{amsmath,amssymb,amscd}
\usepackage{hyperref}
\usepackage{babel}
\usepackage{amstext}
\usepackage{amsmath}
\usepackage{amsfonts}
\usepackage{latexsym}
\usepackage{ifthen}
\usepackage{xypic}
\usepackage[all,cmtip]{xy}
\xyoption{all}
\usepackage{tikz-cd}
\usepackage{graphicx}

\usepackage{enumitem}
\setlist[enumerate]{label=(\thethm.\arabic*), before={\setcounter{enumi}{\value{equation}}}, after={\setcounter{equation}{\value{enumi}}}}

\usepackage{thmtools,bm}
\usepackage{cleveref} 
\usepackage[T1]{fontenc}
\usepackage{newtxtext}
\usepackage{newtxmath}
\usepackage{textcomp}
\theoremstyle{plain}
\newlist{thmlist}{enumerate}{1}
\setlist[thmlist]{wide = 0pt, labelwidth = 2em, labelsep*=0em, itemindent = 0pt, leftmargin = \dimexpr\labelwidth + \labelsep\relax, noitemsep,topsep = 1ex, font=\normalfont, label=(\roman*), ref=\thethm.(\roman{thmlisti})}

\setlist[itemize]{wide = 0pt, labelwidth = 2em, labelsep*=0em, itemindent = 0pt, leftmargin = \dimexpr\labelwidth + \labelsep\relax, noitemsep,topsep = 1ex,} 
\usepackage[text={155mm,251mm},centering, marginparwidth=75pt]{geometry}

\addtotheorempostheadhook[thm]{\crefalias{thmlisti}{thm}}

\addtotheorempostheadhook[lemme]{\crefalias{thmlisti}{lemme}}

\renewcommand{\P}{\mathbb{P}}
\newcommand{\R}{\mathbb{R}}

\newcommand{\Q}{\mathbb{Q}}

\newcommand{\Z}{\mathbb{Z}}

\newcommand{\ep}{\varepsilon}
\renewcommand{\epsilon}{\varepsilon}

\renewcommand{\leq}{\leqslant}
\renewcommand{\geq}{\geqslant}

\newcommand{\vol}{\mathrm{vol}}

\newcommand{\norm}[1]{\left\lVert#1\right\rVert}
\newcommand{\eul}{\genfrac \langle \rangle {0pt}{}}

\newtheorem{thm}{Theorem}[section]
\newtheorem{lemme}[thm]{Lemma}

\newtheorem{proposition}[thm]{Proposition}

\newtheorem{defn}[thm]{Definition}
\newtheorem{cor}[thm]{Corollary}
\theoremstyle{remark}

\newtheorem{remark}[thm]{Remark}

\numberwithin{equation}{thm}

\begin{document}

\title{Invariants of big classes from asymptotic total Betti numbers}

\author{Michele Ancona}
\address{Michele Ancona, Laboratoire J.A. Dieudonn\'e,
UMR CNRS 7351,
Universit\'e C\^ote d'Azur, Parc Valrose,
06108 Nice, Cedex 2, France}
\email{michele.ancona@univ-cotedazur.fr}

\author{Junyan Cao}
\address{Junyan Cao, Laboratoire J.A. Dieudonn\'e,
UMR CNRS 7351,
Universit\'e C\^ote d'Azur, Parc Valrose,
06108 Nice, Cedex 2, France}
\email{junyan.cao@univ-cotedazur.fr}

%\author{Michele Ancona and Junyan Cao}
\maketitle

\begin{abstract} We introduce and study a new numerical invariant of big line bundles on normal complex projective varieties, defined in terms of the total Betti number of hypersurfaces associated with high tensor powers of the line bundle. We prove that this invariant is birational and provides an upper bound for the volume of the line bundle. When the big line bundle admits a rational Zariski decomposition, we show that the invariant coincides with the volume. In particular, this holds when the line bundle is ample or is defined on a projective surface.
We also investigate the analogous problem in the real algebraic setting, proving the existence of real algebraic hypersurfaces associated with high tensor powers of a big line bundle whose real vanishing loci have arbitrarily large total Betti numbers.
\end{abstract}

	\section{Introduction}
The aim of this paper is to introduce a new invariant of big line bundles over complex projective algebraic varieties. This invariant arises from the topology of hypersurfaces defined by high-order tensor powers of the line bundles.  Throughout this text, we use the convention that a projective variety is a complex projective algebraic variety.

\begin{defn}\label{def beta}
Let $X$ be a  projective variety of dimension $n$ and let $L$ be a big line bundle on $X$. We define 
$$\beta (L) := \limsup_{m\to +\infty} \frac{1}{m^{n}} \sup_{Z\in |mL|} b_{*} (Z). $$
Here $b_*$ is the sum of Betti numbers and $|mL|$ is the linear system associated with $mL$.
	\end{defn}
	Here and throughout the rest of the paper, we use the additive notation $mL$ to denote the $m$th tensor power of $L$. 
	Unless otherwise specified, Betti numbers are computed using homology with coefficients in $\mathbb{Q}$. In fact, all the results presented here remain valid over an arbitrary coefficient field $K$, with no changes to the proofs.
  As we will see, the normalization by $m^n$  in the Definition \ref{def beta} ensures that $\beta(L)\in\left]0,+\infty\right[$.	 
%(It looks that we need add one argument to show that it is a limit, not only limit sup. It should be the case when $L$ is big...\michele{pour le moments je mets limsup dans la définition. On peut essayer d'avoir un argument pour mettre seulement lim}). 	
%\subsection*{Properties of the $\beta$-invariant} 
\medskip

\noindent A first result is that the $\beta$-invariant is actually a numerical invariant, namely:
\begin{thm}\label{thm: invariance} Let $X$ be a projective variety. Let $L$ and $L'$ be two big line bundles over $X$. Assume that  $c_1 (L) =c_1 (L')$. Then $\beta (L) =\beta (L')$.  
	\end{thm}	
	%\michele{en fait, la preuve de ça on l'a seulement dans le cas big pour le moment. donc ou bien on le met dans l'enoncé ou bien on essaye de trouver une preuve plus générale}

We will show that for any positive integer $k$, one has $\beta(kL)=k^n\beta(L)$, see Proposition \ref{Q line bundles}. This identity allows us to define $\beta$ for big $\Q$-class that is, for classes $\alpha$ in $H^{1,1}(X,\mathbb Q)$ such that a multiple of $\alpha$ is of the form $c_1(L)$ with $L$ big.	We will also see that $\beta$ is a continuous functions on the big $\Q$-class meaning that if a sequence $(\alpha_k)_{k\in \mathbb N}$ of big $\mathbb Q$-classes converges to a big $\mathbb Q$-class $\alpha$, than $\beta(\alpha_k)$ converges to $\beta(\alpha)$, see Proposition \ref{continuity}. Thus, $\beta$ uniquely extends to a continuous function on the cone $ \mathrm{Big}_{\mathrm{NS}} (X):=(H^{1,1}(X, \mathbb Q)\otimes \mathbb R) \cap \mathrm{Big} (X)$  of a projective variety $X$, where $\mathrm{Big} (X)\subset H^{1,1}(X,\R)$ is the cone of the classes which can be represented by a strictly positive current. The main goal of this paper is to study the properties of such function.  

\medskip

\noindent We prove that $\beta$ is a birational invariant, in the following sense:
	\begin{thm}\label{lemma: beta birational} Let $X$ be a normal projective variety of dimension $n$. Let $L$ be a big line bundle over $X$.  Consider a birational morphism
	$\pi : X' \to X $. Then
	$$\beta (\pi^*  L) = \beta  (L) .$$
	\end{thm}

\medskip

%\subsection*{Volume of a line bundle and $\beta$-invariant} 
\noindent We now consider the problem of estimating the $\beta$-invariant. As it turns out, the value of $\beta$ on a big class $L$ is related with the volume of $L$.
Recall that the volume of a line bundle $L$ over a $n$-dimensional projective variety is defined by $$\displaystyle\vol(L)=\limsup_{m\to+\infty}\frac{n!}{m^n}\dim H^0(X,mL).$$
%By the asymptotic Riemann--Roch theorem, we then interpret the value $\beta$ on ample line bundle $L$ as the equality $\beta(L)=\vol(L)$. 
We obtain prove a lower estimate of the $\beta$-invariant of any big line bundle in terms of its volume.
\begin{thm}\label{thm: vol}
	Let $X$ be a normal projective variety and $L$ be a big line bundle over $X$.  
	%Let $A$ be an ample line bundle such that $A- L$ is $\mathbb Q$-effective.  
	Then we have $\vol (L) \leq \beta (L).$
\end{thm}	
When the line bundle $L$ admits a rational Zariski decomposition in the sense of Definition \ref{zariskidecomp}, then we can prove that the previous inequality is actually an equality.
\begin{thm}\label{thm: equality}
	Let $X$ be a normal projective variety and $L$ be a big line bundle over $X$.  If $L$ admits a rational Zariski decomposition in the sense of Definition \ref{zariskidecomp}, then $\beta (L)=\vol(L)$.
\end{thm}	
It is known that if the ring $R (X, L) := \bigoplus_{m\geq 0} H^0 (X, m L)$ is finitely generated, or if $\dim X \leq 2$, then $L$ admits a rational Zariski decomposition \cite{Laz,Nakayama}. In particular, we obtain a complete description of the $\beta$-invariant of any big line bundle over a normal projective surface: for any big line bundle $L$ over a normal projective surface, we have
	$\beta (L) =\vol(L)$.
	
Theorem \ref{thm: equality} yields also the following result for ample line bundles on normal projective varieties:
\begin{cor}\label{cor: ample}
	Let $X$ be a normal projective variety and $L$ be an ample line bundle over $X$.  Then $\beta (L)=\vol(L)$.
\end{cor} 
Corollary \ref{cor: ample}  was proved by D. Wan and D. Zhang \cite{WanZhang2026} when $X=\P^n$ and by X. Pan, D. Zhang and X. Zhang \cite{PanZhangZhang2026} under the assumption that $X$ is a locally complete intersection. It is however worth emphasizing that we use the results of \cite{PanZhangZhang2026} to prove Corollary \ref{cor: ample}.

\medskip

\noindent In Section \ref{sec: real algebraic}, we define analogue invariants for big line bundles defined over $\R$ over a real projective variety. As an application of our techniques, we obtain the following result:
\begin{thm}\label{real existence} Let $L$ be a big line bundle over a $n$-dimensional normal real projective variety $X$, defined over $\R$.  Assume that $X(\R)\neq \emptyset$. Then, for any $m$ large enough, there exists a real algebraic hypersurface $Z_m\in |mL|$ such that $b_*(Z_m(\R))\geq \vol(L)m^n+O(m^{n-1})$.
\end{thm}

\subsubsection*{Why the volume?} Here is an informal explanation of why the volume of a line bundle appears in the estimates of its $\beta$-invariant. We present it in the simplest setting, where $L$ is an ample line bundle on a smooth $n$-dimensional projective variety $X$. If  $Z_m\in |mL|$ is a smooth element in the linear system, then by Lefschetz hyperplane theorem \cite{andreotti} we have that the only asymptotic contribution to $b_*(Z_m)$ is given by the middle Betti number $b_{n-1}(Z_m)$. Thus, one has the equivalences $b_*(Z_m)\sim_m b_{n-1}(Z_m) \sim_m (-1)^{n-1}\chi(Z_m)$ as $m\to+\infty$. The Euler characteristic $\chi(Z_m)$ can be computed as an integral $\int_{Z_m}c_{n-1}(TZ_m)$, while the Chern class  $c_{n-1}(TZ_m)$ can be related to the Chern classes of $TX$ and of $L$ through the adjunction formula. Using these relations and some algebraic properties of the Chern classes, one can  prove the asymptotics $b_*(Z_m)=\int_Xc_1(L)^n m^n+O(m^{n-1})$. On the other hand, asymptotic Riemann--Roch theorem gives the equality $\vol(L)=\int_Xc_1(L)^n$ when $L$ is ample. This does not yet show that $\beta(L)$ equals the volume of $L$, since $\beta$ is defined via a supremum. Nevertheless, it explains why the volume naturally appears in this context.

\section*{Acknowledgement} We thank Sébastien Boucksom, Antonio Lerario and Dingxin Zhang for useful discussions.

\section*{AI Disclosure} No AI has been used for this work.

%Our aim is to show that $\beta (L)= vol (L)$ when $L$ is big. (or to show the Khovanskii-Teissier type inequality for $\beta$.)

\section{Properties of the $\beta$-invariant}
In this section, we establish several basic properties of the  $\beta$-invariant. We prove that it extends naturally to an invariant of big $\mathbb{Q}$-line bundles, and that it is continuous (Section~\ref{sec continuity}), numerical (Section~\ref{sec numerical}), and birationally invariant (Section~\ref{sec birational}).

%Throughout all the text, a projective variety will be a complex projective variety, non necessarily irreducible, of pure dimension
 We first establish some preliminary results that will be used throughout the paper. The first result, which will be used to control the asymptotic Betti numbers of a family of hypersurfaces, is Lemma \ref{MV}. It is essentially a combination of the Mayer--Vietoris sequence and the fact that every algebraic subvariety admits a neighborhood that deformation retracts onto it, as established by the next result.
\begin{lemme}\label{homotopynb}Let $X$ be a projective variety and $Y\subset X$ be a subvariety. Then there exists a neighborhood $U$ of $Y$ in $X$ such that the inclusion $j:Y\hookrightarrow U$ is  strong deformation retract.
\end{lemme}
\begin{proof} This follows from the work of J. Mather \cite{Mather}. Indeed he proved that given a Whitney stratified space $X$, and given $Y$ a closed union of strata of $X$, then $Y$ admits a neighborhood $U$ such that the inclusion $Y\hookrightarrow U$ is a strong deformation retract. Since every complex algebraic variety admits a Whitney stratification for which a given algebraic subvariety is a union of strata (see for example \cite[Chapter 1.7]{GoreskyMacPherson}), this implies the result.
\end{proof}

\begin{lemme}\label{MV} Let $Z_1$ and $Z_2$ be two subvarieties of a projective variety $X$. Then, there is a long-exact sequence
$$\cdots\rightarrow H_k(Z_1\cap Z_2)\rightarrow H_k(Z_1)\oplus H_k (Z_2)\rightarrow H_k(Z_1\cup Z_2)\rightarrow H_{k-1}(Z_1\cap Z_2)\rightarrow\cdots.$$
\end{lemme}
\begin{proof}
Let $U$ be a neighborhood of $Z_1\cap Z_2$ in $Z_1\cup Z_2$ such that the inclusion of $Z_1\cap Z_2$ in $U$ is a strong deformation retract. Such neighborhood is given Lemma \ref{homotopynb}. In particular, $H_k(U)=H_k(Z_1\cap Z_2)$ for any $k$. Let $U_1:=U\cup Z_1$ and $U_2:=U\cup Z_2$. In particular, the inclusions of $Z_1$ and $Z_2$  respectively in $U_1$ and $U_2$  are strong deformation retracts, so that in particular $H_k(U_1)=H_k(Z_1)$ and $H_k(U_2)=H_k(Z_2)$ for any $k$. The Mayer--Vietoris exact sequence applied to the open cover $\{U_1,U_2\}$  of $Z_1\cup Z_2$ implies the result. 
\end{proof}

The following result is the main theorem of \cite{PanZhangZhang2026}.
%The $\beta$ value of an ample line bundle $L$ is controlled by the following theorem, proved in \cite{PanZhangZhang2026}.
\begin{thm}[\cite{PanZhangZhang2026}]\label{thm: perverse bound}
	Let $X$ is projective variety and let $L$ be an ample line bundle. Then $\beta (L)<+\infty$. Moreover, if $X$ is locally complete intersection, then 
	$$\beta (L) =\vol(L) .$$
	\end{thm}
%	\begin{cor}\label{big bound} Let $X$ is projective variety and let $L$ be a big line bundle. Then $\beta(L)<+\infty$.
%	\end{cor}
%		\begin{proof} Let $L$ be a big line bundle and $A$ an ample line bundle so that $L+A$ is ample. By Lemma \ref{order}, we have $\beta(L)\leq \beta(L+A)$. By Theorem \ref{thm: perverse bound}, we have $\beta(L+A)=\vol(L+A)<+\infty$, hence the result.
%		\michele{Attention, il faut améliorer l'argument car on utilise Lemma \ref{order}, mais dans la preuve de Lemma \ref{order} on utilise Corollary \ref{big bound}...en fait on devrait faire une recurrence pour les deux statement. On prouve big bound en dimension $n-1$, et on prouve order en dimension $n$, qui aide à pouver après big bound en dimension $n$.}.
%	\end{proof}
The following lemma shows that the $\beta$-invariant of a big line bundle is always finite, and that it is monotone with respect to effectivity.
%\subsection{Boundedness} 
\begin{lemme}\label{order}
Let $X$ is projective variety, then the following holds:
\begin{enumerate}
	\item if $L$ is a big line bundle, then $\beta(L)<+\infty$; 
	\item let $L_1$ and $L_2$ be two big line bundles. If $L_1\leq L_2$ (that is, $L_2-L_1$ is effective), then $\beta(L_1)\leq \beta(L_2)$.
\end{enumerate}
\end{lemme}
\begin{proof}
We prove both results by induction on the dimension $n$ of $X$. If $n=1$, then this follows directly from Riemann--Roch formula.
Assume both results hold true in dimension $n-1$, and let us prove them in dimension $n$. 

We start by proving the second assertion. Let $E$ be an $\mathbb Q$-effective divisor such that $L_1+E=L_2$. In particular, any divisor $Z_m$ defined by $|mL_1|$ induces a divisor $Z_m\cup mZ_E$, where $Z_E\in |E|$. Remark that  $b_i(Z_m\cup mZ_E)=b_i(Z_m\cup Z_E)$ for any $i$, since by definition the singular homology of an algebraic variety is the singular homology on the underlying (reduced) topological space.

 Let $\epsilon>0$. 
Let $(Z_m)_m$ be a sequence, with $Z_m\in |mL_1|$, so that $b_*(Z_m)\geq \beta(L_1)m^n-\epsilon m^{n}$ for infinitely many $m$. 
%Let $U_m$ be a neighborhood of $Z_m\cap Z_E$ in $Z_m\cup Z_E$ such that the inclusion of $Z_m\cap Z_E$ in $U_m$ is a strong deformation retract. Such neighborhood is given Lemma \ref{homotopynb}. In particular, $b_i(U_m)=b_i(Z_m\cap Z_E)$ for any $i$. Let $U'_m:=U_m\cup Z_m$ and $U''_m:=U_m\cup Z_E$. In particular $U'_m$ and $U''_m$ are open sets of $Z_m\cap Z_E$ that cover $Z_m\cap Z_E$. Moreover, the inclusions of $Z_m$ and $Z_E$  respectively in $U_m'$ and $U_m''$  are strong deformation retracts, so that in particular $b_i(U'_m)=b_i(Z_m)$ and $b_i(U''_m)=b_i(Z_E)$ for any $i$.
 By induction hypothesis, we have $b_*(Z_m\cap Z_E)=O(m^{n-1})$, since $Z_E$ has dimension $n-1$. By Lemma \ref{MV} applied to $Z_m\cup Z_E$, we then have 
$b_*(Z_m\cup Z_E)= b_*(Z_m) +O(m^{n-1}),$ so that $b_*(Z_m\cup Z_E)\geq \beta(L_1)m^n-2\epsilon m^{n}$ for $m$ large enough. This implies $\beta(L_1)\leq \beta(L_2)$.

Let us now prove the first assertion. Let $L$ be a big line bundle and $A$ an ample line bundle so that $L+A$ is ample. By the second assertion in dimension $n$ (that we have just proved), we have $\beta(L)\leq \beta(L+A)$. By Theorem \ref{thm: perverse bound}, we have $\beta(L+A)=\vol(L+A)<+\infty$, hence the result.
\end{proof}

Finally, we recall a standard fact about big line bundles, commonly known as Kodaira's Lemma (see for example \cite[Corollary 2.2.7]{Laz}).
\begin{lemme}\label{kodaira} Let $L$ and $A$ be respectively a big and an ample line bundle over $X$. Then there exists an integer $m\in \mathbb N$ and an effective line bundle $E$ such that $mL=A+E$.
\end{lemme}

		\subsection{Continuity of the $\beta$-invariant}\label{sec continuity}	Here, we show that the definition of the $\beta$-invariant extends naturally to big $\mathbb{Q}$-line bundles. We also prove that the $\beta$-invariant is locally uniformly continuous as a function on the cone of big $\mathbb{Q}$-line bundles.
	\begin{proposition}\label{Q line bundles} Let $X$ be a $n$-dimensional complex projective variety, $L$ be a big line bundle and $k$ be a positive integer. Then $\beta(kL)=k^n\beta(L)$.
	\end{proposition}
	\begin{proof} 
	%First, remark that we can suppose that there exists $c$ large enough so that $cL$ is effective, since otherwise $\beta(kL)=0$ for any $k\in \mathbb N$ and the proposition is trivial. So for now on, we assume that some multiple of $L$ is effective.
	
	By definition of limsup, one has the inequality $\beta(kL)\leq k^n\beta(L)$.  
	For the other inequality, note first that there exists a constant $C (k, L)$, depending only on $k$ and $L$, such that, for every $m\in \mathbb N$, there exists an integer $p$ with $\frac{C (k, L)}{2}\leq p \leq C (k,L)$ such that 
	\begin{itemize}
		\item $m +p$ is divisible by $k$.
		
		\item $H^0 (X, pL) \neq 0$.
		\end{itemize}
	
	Now for each $p$, we fix a non zero section $s_p \in H^0 (X, p  L)$. 
	Then $s_p$ induces  an injection  
	$$ s_p \otimes :  H^0 (X, mL) \to H^0 (X,  (m+p)L ) . $$ 
	A divisor $Z_{m+p}$ in image of such map is thus a reducible divisor consisting of the union of an hypersurface $Z_m$ of  $|mL|$ and $Y_p := \mathrm{Div} (s_p)$. 	
 Moreover, as $\frac{C (k, L)}{2}\leq p \leq C (k,L)$, we know by Lemma \ref{order} that 
	$$b_*(Y_p \cap Z_m) \leq C m^{n-1}$$
	for some uniform constant $C$ depending only on $C (k, L)$. Thus, by Lemma \ref{MV} applied to $Z_m\cup Y_p$, we obtain 
	$$b_* (Z_{m+p})= b_*(Z_m)+O(m^{n-1}).$$
	In particular, for any $\epsilon>0$, we have, for $m$ large enough,
	$$ \frac{1}{m^{n}} \sup_{Z\in |mL|} b_{*} (Z)\leq \frac{1}{m^{n}}\big(  \sup_{Z\in |(m+p)L|} b_{*} (Z)+\epsilon\big) = \frac{1}{(\frac{m+p}{k})^{n}k^n}\left(  \sup_{Z\in |(\frac{m+p}{k})kL|} b_{*} (Z)+\epsilon\right) $$
Note that by choice of $p$, we have that $\frac{m+p}{k}\in \mathbb N$ and that $p$ is uniformly bounded by a constant which is independent of $m$.	Taking the limsup for $m$ on both sides we obtain, for any $\epsilon>0$, that $\beta(L)\leq k^{-n}\beta(kL)+\epsilon$. Hence the result.
	\end{proof}
	
Proposition \ref{Q line bundles} allows us to define the $\beta$-invariant for big $\Q$-line bundles.

\begin{defn}
Let $\mathrm{Big}_{\Q}(X)$ be the set of big $\Q$-line bundles on $X$. Let $L \in \mathrm{Big}_{\Q}(X)$.  We set 
$$\beta(L):=k^{-n}\beta(kL)$$ for any $k$ such that $kL$ is a big line bundle. Proposition \ref{Q line bundles} implies that $\beta(L)$ does not depend on the choice of $k$.
	\end{defn} 
	
\begin{remark} One can easily see that Lemma \ref{order} extends to the $\beta$-invariants of big $\Q$-line bundles as well. More precisely, if $L_1$ and $L_2$ are two big $\Q$-line bundles such that $L_2-L_1$ is $\mathbb{Q}$-effective, then $\beta(L_1)\leq \beta(L_2)$. In the rest of the paper, we will write $L_1\leq L_2$ whenever $L_2-L_1$ is $\mathbb{Q}$-effective.
\end{remark}

\begin{proposition}\label{continuity} Let $X$ be a  projective variety and $\mathrm{Big}_{\Q}(X)$ be the set of big $\Q$-classes of $X$. Then the function $\beta:\mathrm{Big}_\Q(X)\rightarrow \R$ is locally Lipschitz.
\end{proposition}
\begin{proof} Let $L$ be a big $\Q$-line bundles and  we fix an auxiliary norm on $\mathrm{Pic}_{\mathbb \R}(X)$. Then we can find an ample line bundle $A$ on $X$ such that $A - S$ is $\mathbb \Q$-effective for every $S \in \mathrm{Pic}_{\mathbb \Q}(X)$ with $\|S\| \leq 1$.
Since $L$ is big, by Lemma \ref{kodaira}, there exists a constant $c_0 >0$ depending only on $L$ and $A$ such that $L -c_0 A$ is $\mathbb Q$-effective.
	
Now let $L'$ is a $\Q$-line bundles such that $\norm{L'-L}<\epsilon$. 
On one hand,  we have 
$$L'= L + (L'-L) \leq L + \ep A \leq (1+\frac{\ep}{c_0}) L .$$
By Lemma \ref{order} and Proposition \ref{Q line bundles}, we  obtain $\beta(L') \leq \beta((1+\frac{\epsilon}{c_0})L)=(1+\frac{\epsilon}{c_0})^n\beta(L)$, where $n$ is the dimension of $X$. 

On the other hand, we have 
$$L'= L + (L'-L) \geq L - \ep A \geq (1 - \frac{\ep}{c_0}) L .$$
By the same argument as above, we have $\beta(L') \geq (1-\frac{\epsilon}{c_0})^n \beta(L)$.  

Finally, by the choice of $A$, we have $A - \frac{1}{\|L\|} L \geq 0$. By Lemma \ref{order}, we know that $\beta (L) \leq \|L\|  \cdot \beta (A)$. Therefore we obtain 
$$|\beta (L') -\beta (L) | \leq C \ep \beta (L) \leq C \ep  \cdot \|L\|  \cdot \beta (A) .$$
In particular, $\beta$ is locally Lipschitz.
\end{proof}

	%\begin{lemme} Let $L$ be a big line bundle over $n$-dimensional projective manifold $X$ and  $Y$  be a subvariety of $X$. Let  $(Z_m)_m$, with $Z_m\in |mL|$, be a sequence verifying $\limsup_{m\to +\infty} \frac{1}{m^n}b_*(Z_m)>0$,  then $b_*(Z_m)\sim_{m\to+\infty}b_*(Z_m\setminus (Z_m\cap Y))$.

%	\end{lemme}
	
%	\begin{proof} Let $U_m$ be a neighborhood of $Z_m\cap Y$ in $Z_m$ that retracts by deformation onto $Z_m\cap Y$. By Lemma \ref{order}, we have $b_*(U_m)=O(m^{n-1})$. We also claim that $b_*(U_m\cap (Z_m\setminus (Z_m\cap Y)))=O(m^{n-1})$. Remark that this implies the result, since if we write $Z_m=U_m\cup (Z_m\setminus (Z_m\cap Y))$ and apply Mayer--Vietoris, we obtain exactly $b_*(Z_m)=b_*(Z_m\setminus (Z_m\cap Y))+O(m^{n-1})$. Let us now prove the claim. For this, remark that we can assume that $Z_m\cap Y$ is a simple normal crossing divisor. Indeed, if not, we can solve the singularities by blow-ups, and obtain a map $\pi:\tilde{Z}_m\to Z_m$ with an exceptional divisor $E$ over $Z_m\cap Y$ which is a simple normal crossing divisor. Now, the map $\pi$ is an isomorphism outside $E$ so that $\tilde{Z}_m\setminus E\cong Z_m\setminus Z_m\cap Y)$. Now, $U_m\cap (Z_m\setminus (Z_m\cap Y))$ retract by deformation onto $\partial U_m$, which is a sphere bundle over $Z_m\cap Y$. 
%	\michele{a completer. }
%		\end{proof}

\subsection{Numerical invariance of the $\beta$-invariants}\label{sec numerical}
	Here we prove Theorem \ref{thm: invariance}, namely that if two  big line bundles  $L$ and $L'$ over $X$ verify  $c_1 (L) =c_1 (L')$, then $\beta (L) =\beta (L')$. 
%We now recall the statements of Theorems \ref{thm: invariance} and \ref{thm: real invariance} in the following:
%\begin{thm}\label{thm: invariant}
%	Let $X$ be a  projective variety. Let $L$ and $L'$ be two big line bundles over $X$. Assume that  $c_1 (L) =c_1 (L')$. Then $\beta (L) =\beta (L')$. 
	% If we assume moreover that $X$ is defined over $\R$ and $L$ and $L'$ are defined over $\R$ too, then $\beta_\R (L) =\beta_\R (L')$ and $\beta_{i, \R} (L) =\beta_{i,\R} (L')$.
%	\end{thm}	
	\begin{proof} Let $L$ and $L'$ be two big line bundles such that $c_1 (L) =c_1 (L')$.
		
				Let $A$ be an ample line bundle on $X$ such that $A + \rho$ is effective for every $\rho \in \mathrm{Pic}^0 (X)$. Then every $m\in \mathbb N$, we know that 
	$$(m L  +A ) -  m L' = A+ (m L- mL') $$
is effective. 
By Lemma \ref{order}, we know that $\beta (m L +A) \geq \beta (m L')$. Together with Proposition \ref{Q line bundles}, we obtain
$$\beta (L +\frac{1}{m} A) \geq \beta (L') .$$
By letting $m\to +\infty$, together with Proposition \ref{continuity}, we obtain
$$\beta (L ) \geq \beta (L') .$$
By symmetry, we obtain thus $\beta (L) =\beta (L')$.
	\end{proof}
	
	\begin{remark}
			By Proposition \ref{continuity} and Theorem \ref{thm: invariance}, $\beta$ induces a continuous function 
			$$\beta: \mathrm{Big}_{\mathrm{NS}} (X):=(H^{1,1}(X, \mathbb Q)\otimes \mathbb R) \cap \mathrm{Big} (X)\rightarrow \R ,$$
			where $\mathrm{Big} (X)\subset H^{1,1}(X,\R)$ is the cone representing the classes which can be represented by a strict positive current.
						\end{remark}

\subsection{Birational invariance}\label{sec birational}	In this section, we prove Theorem \ref{lemma: beta birational}, namely that $\beta$ is a biratioal invariant: for any modification $\pi : X' \to X $ we have $\beta(\pi^*L)=\beta(L)$.
	\begin{proof}[Proof of Theorem \ref{lemma: beta birational}] Let $Y\subset X$ be a subvariety of $X$ such that $\pi:X'\setminus E\to X\setminus Y$ is an isomorphism, where $E=\pi^{-1}(Y)$ is the exceptional divisor.
	Let $Z_m$ be a divisor in the linear system $|mL|$ and let $Z_m'$ be the divisor in the linear system $|\pi^*mL|$ corresponding to $Z_m$ by the pull-back isomorphism $H^0(X,mL)\to H^0(X',\pi^*mL)$ (such pull-back is an isomorphism since $X$ is normal). Assume that $\limsup_{m\to +\infty} \frac{1}{m^n}b_*(Z_m)>0$. In order to prove the result, it suffices to prove that $\limsup_{m\to +\infty} \frac{1}{m^n}b_*(Z_m)=\limsup_{m\to +\infty} \frac{1}{m^n}b_*(Z'_m)$. 
		
	In order to do this,
	consider the following open cover $Z_m=U_m \cup (Z_m\setminus (Z_m\cap Y))$, where $U_m$ is a small open neighborhood of $Z_m\cap Y$ which retract by deformation onto $Z_m\cap Y$.  We also consider the open cover $Z'_m=U'_m \cup (Z'_m\setminus (Z'_m\cap E))$. We can assume that $U'_m=\pi^{-1}(U_m)$ and that it retracts by deformation onto $E\cap Z'_m$.  Indeed, remark that $Z_m\cap Y$ can be realized as the zero locus of a semialgebraic positive function $f$ defined on $Z_m$ (for instance, consider a real algebraic embedding of $Z_m$ inside some $\R^N$, and take $f$ to be the Euclidian distance from $Z_m\cap E$). Then, we can take $f\circ \pi$, which is a semialgebraic positive function on $Z'_m$ which vanishes precisely on $Z'_m\cap E$. Now, for $\epsilon$ small enough, the open sets $\{f<\epsilon\}=:U_m$ and $\{f\circ \pi<\epsilon\}=:U'_m$ retracts by deformation respectively on $\{f=0\}$ and $\{f\circ \pi=0\}$ (see for example \cite[Theorem 1.47]{lerario}).
	%(\michele{Si Y est lisse c'est clair, car $U'_m$ est un petit voisinage de la section nulle si on regarde le diviseur exceptionel comme section nulle du fibré tautologique }).
	
Now	we apply Mayer--Vietoris sequence to these two covers of $Z_m$ and $Z_m'$ to obtain the following two long exact sequences.
	$$H_{k+1}(Z_m)\to H_{k}(U_m \cap (Z_m\setminus (Z_m\cap Y)))\to H_k(U_m)\oplus H_{k}( Z_m\setminus (Z_m\cap Y)) \to H_{k}(Z_m);$$
	
$$H_{k+1}(Z'_m)\to H_{k}(U'_m \cap (Z'_m\setminus (Z'_m\cap E)))\to H_k(U'_m)\oplus H_{k}( Z'_m\setminus (Z'_m\cap E)) \to H_{k}(Z'_m).$$

	By Lemma \ref{order}, the dimension of $H_k(U_m)$ and $H_k(U_m')$ is $O(m^{n-1})$ since $U_m$ and $U_m'$ retract by deformation onto $Y\cap Z_m$ and $E\cap Z_m'$ and the dimensions of $Y$ and $E$ are at most $n-1$.

	On the other hand, we have canonical isomorphisms  $H_k(U_m \cap (Z_m\setminus (Z_m\cap Y)))\cong H_k(Z'_m\setminus (Z'_m\cap E)))$ and $H_k( Z_m\setminus (Z_m\cap Y))\cong H_k(Z'_m\setminus (Z'_m\cap E))$ induced by $\pi$, since $\pi$ is an isomorphism outside $E$. Hence,  $\dim H_k(Z_m)=\dim H_k(Z'_m)+O(m^{n-1})$ so that $b_*(Z_m)=b_*(Z'_m)+O(m^{n-1})$. Hence the result.
		\end{proof}

	\section{Control of the $\beta$-invariants}
	In this section, we establish bounds for the $\beta$-invariant of a big line bundle. In particular, we prove that it is always bounded from below by the volume of the line bundle, and that equality holds whenever the line bundle admits a rational Zariski decomposition in the sense of Definition \ref{zariskidecomp}. It is well-known that such a decomposition exists when the section ring is finitely generated or the variety has dimension two (see Propositions \ref{zariski for fg} and \ref{surface}). In particular, we prove Theorems  \ref{thm: vol} and \ref{thm: equality}, stated in the introduction. To do so, we first recall some properties of the volume of a big line bundle and of the rational Zariski decomposition of a big line bundle.
	
		\begin{defn}\label{zariskidecomp} Let $L$ be a big line bundle over $X$. We say that $L$ admit a rational Zariski decomposition, if there exists a modification $\pi:\widetilde{X}\to X$, with $\widetilde{X}$ smooth, and two $\mathbb Q$-effective divisors $P$ and $N$ over $\widetilde{X}$ such that:
		\begin{enumerate}
			\item $\pi^*L= P+N$;
			\item $P$ is big and nef, and $N= \sum_{i=1}^\ell a_i N_i$ is $\mathbb Q$-effective.
			\item for any $m\in \mathbb N$, the natural map 
			$$H^0(\widetilde{X}, m\pi^*L- \lceil m N\rceil) \to H^0(\widetilde{X}, m\pi^*L) $$ induced by the multiplication by $\bigotimes_{i=1}^{\ell}s_i ^{\otimes \lceil a_i m \rceil}$ is an isomorphism. Here, $s_i$ are sections defining $N_i$ and $\lceil\cdot\rceil$ denotes the round-up integer.
		\end{enumerate}
	\end{defn}

	\begin{proposition}\cite[III, Remark 1.17]{Nakayama}\label{zariski for fg} Let $L$ be a big line bundle over $X$ such that $R (X, L) := \bigoplus_{m\geq 0} H^0 (X, m L)$ is finitely generated.  $L$ admits a rational Zariski decomposition. 
\end{proposition}

	\begin{proposition}\cite[Prop 2.3.21]{Laz}\label{surface}
		Let $L$ be a big line bundle on a projective surface $X$. Then $L$ admits rational Zariski decomposition.
	\end{proposition}
	A rational Zariski decomposition does not exist in general. The idea of Fujita \cite{Fujita1994} is to consider an approximate Zariski decomposition of a big line bundle. In what follows, we will need the following version of such an approximate Zariski decomposition.
	\begin{thm}[Approximate Zariski decompostion \cite{Fujita1994, Boucksom2004, Demailly2010}]\label{approximate}
		Let $L$ be a big line bundle on a projective variety $X$. Then for every $\ep >0$, there exists a birational morphism
		$$\pi_\ep : X' \to X ,$$
		such that 
		$$\pi^* _\ep L = A + \sum a_i E_i ,$$ 
		where $A$ is $\mathbb Q$-ample, $a_i >0$ and $\vol (A) \geq \vol (L) -\ep$.
		\end{thm}
		
		\begin{proof}
			By using \cite{Fujita1994}, for every $\ep >0$, there exists a birational morphism
			$$\pi_\ep : X' \to X ,$$
			such that 
			$$\pi^* _\ep L = A + \sum a_i E_i ,$$ 
			such that $A$ is big and nef and $\vol (A) \geq \vol (L) -\ep$.  
			
	As $\pi^* _\ep L$ is big, we can write $\pi^* _\ep L = A' +E'$ where $A'$ is ample, $E'$ is $\mathbb Q$-effective.  Then 
	$$\pi^* _\ep L = ((1-\ep) A + \ep A' ) + ((1-\epsilon)\sum a_i E_i  +\epsilon E')$$
	is the desired decomposition.
			\end{proof}

\begin{proof}[Proof of Theorem \ref{thm: equality}] By Definition \ref{zariskidecomp}, one can find a modification $\pi:\widetilde{X}\to X$ and two line bundles $P$ and $N$ verifying the three conditions stated in Definition \ref{zariskidecomp}.  We claim that $\beta(\pi^*L)=\beta(P)$. In order to prove this, it is enough to prove that 
\begin{equation}\label{equal limsup}\limsup_m \frac{b_*(Z_m)}{m^n}=\limsup_m \frac{b_*(Z'_m)}{m^n},
\end{equation} 
where  $(Z_m)_m$ is a sequence of hypersurfaces with $Z_m\in |m P|$, and $Z_m'=Z_m\cup Z_N$ is the hypersurface in the linear system $|m\pi^*L|$, induced by the isomorphism given by the third property of Definition \ref{zariskidecomp}. Now, the equality \eqref{equal limsup} holds thanks to Lemma \ref{MV}  applied to 
$Z_m\cup Z_N$, and by remarking that, by Lemma \ref{order}, we have $b_*(Z_m\cap Z_N)=O(m^{n-1})$.

We have then proved that $\beta(\pi^*L)=\beta(P)$. Remark moreover, that by Definition \ref{zariskidecomp} we have  $\vol(L)=\vol(P)$.
Since $P$ is big and nef, one has that for any $\epsilon>0$ small enough and any $A$ ample, the divisor $P+\epsilon A=:P_\epsilon$ is ample. By Theorem \ref{thm: perverse bound}, since $\widetilde{X}$ is smooth and $P_\epsilon$ is ample, we have that $\beta(P_\epsilon)=\vol(P_\epsilon)$. By continuity of the $\beta$-invariant and of the volume, one then has that $\beta(P)=\vol(P)$. By Theorem \ref{lemma: beta birational}, we have $\beta(\pi^*L)=\beta(L)$. Hence the result.
\end{proof}

\begin{proof}[Proof of Theorem \ref{thm: vol}]
We consider the approximate Zariski decomposition given by Theorem \ref{approximate}: for every $\ep >0$, there exists a birational morphism
$$\pi_\ep : X' \to X ,$$
such that $\pi^* _\ep L = A + \sum a_i E_i$, where $A$ is ample, $\vol (A) \geq \vol (L) -\ep$.
%and $$\sum a_i A^{n-1} E_i \leq \ep .$$
By Lemma \ref{order}, we have $\beta (\pi^* _\ep L) \geq \beta  (A)$, while by Theorem \ref{thm: perverse bound} we have $\beta(A)=\vol(A)$.
This gives us
$\beta (\pi^* _\ep L)  \geq \vol (L) -\ep.$
By Lemma \ref{lemma: beta birational}
$$\beta (\pi^* _\ep L) = \beta  (L) .$$
This implies $\beta  (L) \geq  \vol (L) -\ep$ for every $\ep >0$. So the theorem is proved.
		\end{proof}
		
\begin{remark}	
The $\beta$-invariant can be defined in the same way for any line bundle, not only for big ones, using the same formula. We can see that by definition $\beta(L)=0$ when $L$ is not effetive.
The previous results led to  natural questions: given a line bundle $L$ over $X$, is it true that  $\beta (L) >0$ if and only if $L$ is big? Is the equality $\vol (L) = \beta (L) $ always true?
                Recall that by Theorem  \ref{thm: equality}, the anwer to the last question is positive when $L$ admit a Zariski decomposition and $X$ is normal.
               \end{remark}	
		
\section{$\beta_\R$-invariants and real algebraic geometry}\label{sec: real algebraic}	
In this section, we introduce the $\beta$-invariant in the real algebraic context.
A real projective variety $X$ is a projective algebraic variety defined over $\R$. Given a real projective variety $X$, we denote by $X(\R)$ its real locus. 
\begin{defn}
Let $X$ be a projective variety and let $L$ be a holomorphic line bundle on $X$, both defined over $\R$. We define 
$$\beta_{\mathbb R} (L) := \limsup_{m\to +\infty} \frac{1}{m^n} \sup_{Z\in |mL|} b_{*} (Z (\mathbb R)) $$
and 
$$\beta_{i, \mathbb R} (L) := \limsup_{m\to +\infty} \frac{1}{m^n} \sup_{Z\in |mL|} b_{i} (Z (\mathbb R)).$$
Here $b_*$ is the sum of Betti numbers.
	\end{defn}

Many of the properties we proved for the $\beta$-invariant also hold for the $\beta_\R$- and $\beta_{i,\R}$-invariants. For example, these invariants are numerical, in the sense that if $c_1(L)=c_1(L')$, then $\beta_\R(L)=\beta_\R(L')$ and $\beta_{i,\R}(L)=\beta_{i,\R}(L')$. They are also birational invariants and satisfy $\beta_\R(kL)=k^n\beta_\R(L)$ and  $\beta_{i,\R}(kL)=k^n\beta_{i,\R}(L)$, so that they extend naturally to $\Q$-big real line bundles, and this extension is locally uniformly continuous. The proofs of these facts are identical to those in the complex case and are therefore omitted. 

In the remainder of this section, we estimate the invariants $\beta_\R$ and $\beta_{i,\R}$. We begin by briefly discussing the relationship between the topology of the real locus of a real projective variety and that of its complex locus.
Let $X$ be a real projective variety. Smith--Thom inequality \cite{Thom,Harnack} says that the total Betti number with $\mathbb{Z}/2$ coefficients of $X$ is always bigger to or equal than the total Betti number with $\mathbb{Z}/2$ coefficient of its real locus $X(\R)$, that is, $b_*(X(\R),\mathbb{Z}/2)\leq b_*(X,\mathbb{Z}/2)$. A variety verifying the equality is called maximal. Given a real projective variety $X$ and a line bundle $L$ defined over $\R$, a classical problem in real algebraic geometry is to find maximal hypersurfaces $Z\in |L|$. This is not always possible: for example if $C$ is a  real algebraic curve with $b_0(C(\R))\leq g(C)-1$ (e.g. a smooth genus $3$ curve whose real locus is connected), than the surface $X=C\times C$ cannot contain any maximal curve, see for example \cite[Corollary 5.6]{ABW}.

A more flexible problem is to study asymptotic maximality, that is, to find $Z_m\in |mL|$ such that $$\displaystyle\limsup_{m\to+\infty}\frac{b_*(Z_m(\R))}{b_*(Z_m)}=1.$$
If $L$ is ample and $X$ is smooth, it was proven in \cite{ABW} that one can always find such asymptotically maximal hypersurfaces (see also \cite{ItenbergViro2006} and \cite{Bertrand2006} for the case of projective space and toric varieties). In other terms, if $X$ is a real projective manifold, and $L$ be an ample line bundle on $X$ defined over $\mathbb R$, then $\beta_{\mathbb R}(L)\geq\vol(L)$. The following result proves that this inequality is actually an equality.
	\begin{proposition}\label{thm: beta for R}
	Let $X$ be a smooth real projective variety and let $L$ be an ample line bundle defined over $\mathbb R$. Then 
	$$\beta_\R (L) =\vol(L) .$$
	\end{proposition}
\begin{proof}  Remark that we can define $\beta(L,\Z/2)$ and $\beta_\R(L,\Z/2)$ in the same fashion as $\beta(L)$ and $\beta(L,\Z/2)$, by taking $\Z/2$ coefficients instead of integer coefficients. In this case, one still has $\beta(L)=\beta(L,\Z/2)=\vol(L)$ when $L$ is ample, since the total $\Z/2$-Betti number of smooth hypersurfaces $Z_m$ of $|mL|$ grows like $\vol(L)m^n$ (see \cite{ABW}) and since,  by  \cite{PanZhangZhang2026}, one has  $\beta(L,\Z/2)\leq \vol(L)$. 

On the other hand, by Smith--Thom inequality, one has $\beta(L,\Z/2)\geq \beta_\R(L,\Z/2)$. On the other hand, by  Theorem \cite{ABW} we know that $\beta_\R(L)\geq \vol(L)$. Putting everything together, we get $\vol(L)=\beta(L,\Z/2)\geq \beta_\R(L,\Z/2)\geq \beta_\R(L)\geq \vol(L)$, hence the result.
\end{proof}

	The following  is the real algebraic analogue of Theorem \ref{thm: vol}. 
	
	\begin{thm}\label{real} 
	Let $X$ be a real normal projective variety and $L$ be a big line bundle on $X$ defined over $\R$.  
	%Let $A$ be an ample line bundle such that $A- L$ is $\mathbb Q$-effective.  
	Then we have $\vol (L) \leq \beta_\R (L)$.

	\end{thm}	
Theorem \ref{real existence} in the introduction is a direct consequence of Theorem \ref{real}, the proof of which follows the one of Theorem \ref{thm: vol}, and it is sketched below. For this we need the following approximate  decomposition over the reals: 
\begin{thm}[Approximate Zariski decomposition over the reals]\label{zariski over R} 
	Let $X$ be a real normal projective variety and $L$ be a big line bundle on $X$ defined over $\R$.  
	Then for every $\ep >0$, there exists a birational morphism
$$\pi_\ep : X' \to X ,$$
such that $X'$ is a real smooth projective variety, 
$$\pi^* _\ep L = A + \sum a_i E_i ,$$ 
where $A$ is $\mathbb Q$-ample defined over $\R$, $a_i \in \R_{>0}$, and $\vol (A) \geq \vol (L) -\ep$.	
\end{thm}

\begin{proof}  Let $m\in \mathbb N$.  Let $B_m$ be the base locus of $|mL|$. Since $L$ is defined over $\R$, its base locus is also defined over $\R$. Up to blow-up $B_m$, we can find a  birational morphism $\pi_m: X_m \to X$ such that $X_m$ is a smooth projective manifold defined over $\R$ and the base locus $Y_m:= \mathrm{Base}(m \pi_m ^* L)$ is normal crossing. Then we have 
	$$\pi_m ^* L = A_m + \frac{1}{m }Y_m ,$$
	where $A_m$ is a $\mathbb Q$-semipositive and big divisor defined over $\R$. Moreover we know that $\lim_{m\to +\infty }\vol (A_m) = \vol (L) $. 
	
	Now let $\widetilde{A}$ be an ample line bundle defined over $\R$ on $X_m$. Since $A_m$ is big, there exists a $c_0 > 0$ such that $A_m - c_0 \widetilde{A}$ is $\mathbb Q$-effective. Then for every $\delta_m\in\mathbb Q^+$, we have the decompoistion
		$$\pi_m ^* L = A_m + \frac{1}{m }Y_m = ((1-\delta_m)A_m + c_0 \delta_m \widetilde{A} )+ \delta_m (A_m -c_0 \widetilde{A}) + \frac{1}{m }Y_m. $$
	Set $\widetilde{A}_m := ((1-\delta_m)A_m + c_0 \delta_m \widetilde{A} )$ and $E_m:= \delta_m (A_m -c_0 \widetilde{A}) + \frac{1}{m }Y_m$. Then we have
	$$\pi_m ^* L  = \widetilde{A}_m +E_m .$$
	By construction, we know that $\widetilde{A}_m$ is $\mathbb Q$-ample divisor on $X_m$ defined over $\R$ and $E_m$ is $\mathbb Q$-effective.  Moreover the volume of $\widetilde{A}_m$ is larger than $(1-\delta_m)^n \vol (A_m)$. The theorem is thus proved by taking $\delta_m \to 0$ when $m\to +\infty$.
\end{proof}
	\begin{proof}[Proof of Theorem \ref{real}]
	The proof follows the same line as Theorem \ref{thm: vol}. We consider the approximate Zariski decomposition over the reals as in Theorem \ref{zariski over R}: for every $\ep >0$, there exists a birational morphism defined over $\R$
$$\pi_\ep : X' \to X ,$$
such that $\pi^* _\ep L = A + \sum a_i E_i$, where $A$ is ample, $\vol (A) \geq \vol (L) -\ep$.
%and $$\sum a_i A^{n-1} E_i \leq \ep .$$
The exact same proof of Lemma \ref{order} works over the reals, so that we have $\beta_\R (\pi^* _\ep L) \geq \beta_\R  (A)$. On the other hand, by Proposition \ref{thm: beta for R} we have $\beta_\R(A)=\vol(A)$.
This gives us
$\beta_\R (\pi^* _\ep L)  \geq \vol (L) -\ep.$
Similarly, restricting the proof of Lemma \ref{lemma: beta birational} to the real loci, we have
$$\beta_\R (\pi^* _\ep L) = \beta_\R  (L) .$$
This implies $\beta_\R  (L) \geq  \vol (L) -\ep$ for every $\ep >0$. So the inequality $\vol (L) \leq \beta_\R (L)$ is proved. 
	\end{proof}
	\begin{remark} By combining the techniques of \cite{ABW} with the results of \cite{Arnal2022,Bihan2003}, one can prove that, if $A$ is an ample line bundle on a smooth  real projective variety $X$, then $\beta_{i,\R}(A)>\eul{n}{i}\frac{1}{n!}\vol(A)$, where $\eul{n}{i}$ denotes the Eulerian number counting the permutations of $\{1,\dots,n\}$ with $i$ descents. In this case, since $\sum_i\eul{n}{i}=n!$, this implies that the sum of the $\beta_{i,\R}(A)$'s is stricly bigger than $\beta_{\R}(A)$.
	Moreover, using the inequality  $\beta_{i,\R}(A)>\eul{n}{i}\frac{1}{n!}\vol(A)$ and following the same lines as in the proof of Theorem \ref{real}, one can prove that the inequality $\eul{n}{i}\frac{1}{n!}\vol (L) < \beta_{i,\R} (L)$ holds for any big line bundle on a normal real projective variety $X$. 
		 	\end{remark}

 \bibliography{biblio}
\bibliographystyle{abbrv}
\end{document}